%% file: nh.tex
\documentclass[reqno]{amsart}

\input{definitions.tex}
\excludecomment{maximacode}
\excludecomment{arxivabstract}

\begin{document}

\title[Invariant tori]{Invariant tori for the Nos{\'e} Thermostat near the High-Temperature Limit}
\author{Leo T. Butler}
\address{Department of Mathematics, North Dakota State University, 
Fargo, ND, USA, 58108}
\email{leo.butler@ndsu.edu}
\date{\today}

\begin{abstract}
  Let $H(q,p) = \frac{1}{2} p^2 + V(q)$ be a $1$-degree of freedom mechanical
  Hamiltonian with a $C^r$ periodic potential $V$ where $r>4$. The Nos{\'e}-thermostated
  system associated to $H$ is shown to have invariant tori near the
  infinite temperature limit. This is shown to be true for all thermostats
  similar to Nos{\'e}'s. These results complement the result of Legoll,
  Luskin and Moeckel who proved the existence of such tori near the
  decoupling limit~\cite{MR2299758,MR2519685}.
\end{abstract}
\begin{arxivabstract}
  Let H(q,p) = ½p² + V(q) be a 1-degree of freedom mechanical
  Hamiltonian with a Cⁿ periodic potential V where n>4. The
  Nosé-thermostated system associated to H is shown to have invariant
  tori near the infinite temperature limit. This is shown to be true
  for all thermostats similar to Nosé's. These results complement the
  result of Legoll, Luskin and Moeckel who proved the existence of
  such tori near the decoupling limit.
\end{arxivabstract}

\maketitle

\section{Introduction} \label{sec:intro}

The computation of equilibrium statistical properties of molecular
systems is of great importance to applied subjects such as biology,
chemistry, computational physics and materials science. These
equilibrium statistical properties are phase space integrals like
\begin{equation}
  \label{eq:eq-stat}
  \mean{f} = \int f(q,p) \D{μ}, \qquad \D{μ} = \exp(- β H) \D{p}\D{q}/Z,
\end{equation}
where $q$ is the position of the system and $p$ is its momentum, $H =
H(q,p)$ is the total energy of the system, $\beta=1/T$ is the
reciprocal of the equilibrium temperature $T$ and $Z = Z(β)$ is a
normalization constant, also called the partition function.

In practice, $f=f(q,p)$ is a ``measurement'' or ``observable'', such
as the position of the first atom in the system. The computation of
the integral \eqref{eq:eq-stat} can be very expensive, so one often
wants to replace that multi-dimensional average with the time average
\begin{equation}
  \label{eq:birkhoff-average}
  \bmean{f} = \lim_{T \to \infty} \dfrac{1}{T} \int_0^T f(q(t),p(t)) \D{t}
\end{equation}
where $(q(t),p(t))$ are the position and momenta of the system at time
$t$. In principle, $\bmean{f}$ depends on the initial condition
$(q(0),p(0)$. When, for almost all initial conditions the average in
\eqref{eq:birkhoff-average}--called a Birkhoff average--converges to
$\mean{f}$ the system is \defn{ergodic}. Ergodic systems have many
interesting properties, but from the point-of-view here, they provide
a means to an end: reduction of the multi-variable integral
\eqref{eq:eq-stat} to a single-variable integral.

In equilibrium statistical mechanics, the Hamiltonian $H$ is the
internal energy of an infinitesimal system $S$ that is immersed in a
heat bath $B$ at the temperature $T$. A simple model of the exchange
of energy between the infinitesimal system $S$ and heat bath $B$ was
introduced by Nos{\'e} \cite{nose}. This consists of adding an extra
degree of freedom $s$ and rescaling momentum by $s$:
\begin{equation}
  \label{eq:nose}
  F = H(q,p s^{-1}) + \underbrace{\dfrac{1}{2 M} p_s^2 + n k T \ln s}_{N},
\end{equation}
where $n$ is the number of degrees of freedom of the system $S$, $M$
is the mass of the thermostat and $k$ is Boltzmann's
constant. Nos{\'e}'s thermostated Hamiltonian $F$ has two desirable
properties: the orbit average of $\temp{} = \norm{p s^{-1}}^2$ is $T$
and the thermostated system is Hamiltonian. A drawback of the Nos{\'e}
thermostat is the measure $\D{μ}_N = \exp(-β F)
\D{p}\D{q}\D{p_s}\D{s}$ is not normalizable (i.e. there is no
partition function for $F$), so phase space averages with respect to
the extended phase space variables $(q,p,s,p_s)$ are undefined.

Hoover~\cite{hoover} introduced a non-symplectic reduction of Nos{\'e}'s
thermostat by eliminating the state variable $s$ and rescaling
time $t$:

\[ q = q,\qquad ρ = ps^{-1},\qquad \ddt{τ} = s \ddt{t},\qquad ξ = \ddt[s]{τ}. \]

This reduction has the desirable properties: when $E = H(q,ρ) +
\frac{1}{2M} ξ^2$, the measure $\D{μ}_E = \exp(-β E) \D{q}\D{ρ}\D{ξ}$
is finite and so has a partition function; it projects to $\D{μ}$
\eqref{eq:eq-stat}; it is stationary for the reduced thermostat; and when
the system is a simple harmonic oscillator, the equilibrium
statistical mechanical model predicts the variates $q, ρ$ and $ξ$ are
Gaussian.

Indeed, the Nos{\'e}-Hoover thermostated simple harmonic oscillator
reduces to the following ``simple'' system:
\begin{align}
  \label{eq:nose-hoover}
  \dot{q} &= ρ, && \dot{ρ} = -q - \xi ρ, &&
  \dot{\xi} = \left( ρ^2 - T \right)/M.
\end{align}
Legoll, Luskin and Moeckel show in~\cite{MR2299758} that near the
decoupled limit of $M=∞$ and $ξ=0$, the thermostated harmonic
oscillator \eqref{eq:nose-hoover} is non-ergodic. By means of an
averaging argument, they reduce the thermostated equations to a
non-degenerate twist map to show the existence of KAM tori. The result is
generalized in a subsequent paper to $1$-degree of freedom thermostats
for which an associated potential function $G$ (eq. 33 of
\cite{MR2519685}) is not isochronous.

\subsection{The high-temperature limit}
\label{sec:high-t-limit}

The present paper examines the dynamics of Nos{\'e}'s thermostat near
the high-temperature limit $T=∞$ with the thermostat mass $M$ held
constant. It presents a proof of the existence of KAM tori based on
the integrability of suitably rescaled equations at the $T=∞$
limit. Specifically,

\begin{theorem}
  \label{thm:kam-tori-high-temperature-limit}
  Let $V : \R/2 π \Z → \R$ be $C^r$, $r>4$, and let $H : \cotangent \R/2 π \Z → \R$ be
  \begin{equation}
    \label{eq:h}
    H(q,p) = \frac{1}{2}  p^2 + V(q).
  \end{equation}
  Fix the thermostat mass $M > 0$. The Nos{\'e}-thermostated
  Hamiltonian $F$ \eqref{eq:nose} associated to $H$ possesses
  invariant KAM tori for all $T > 0$ sufficiently large.
\end{theorem}

The intuition behind this theorem is the following: for large
temperatures, because the potential $V$ is bounded, most of the energy
must be kinetic. Therefore, the dynamics should look like a
perturbation of the purely kinetic hamiltonian (where $V \equiv
0$). While this picture is not exactly correct, it is accurate in that
the high-temperature thermostated system resembles a perturbation of
an integrable system.

\subsection{Alternative Thermostats}
\label{sec:alt-thermostats}

A natural question that arises in light of the above results on the
existence of invariant tori is whether there are thermostats like
Nos{\'e}'s that do not possess these invariant KAM tori in the large
temperature limit. Let's say that a Nos{\'e}-like thermostat is one
which involves momentum rescaling and the thermodynamic equilibrium
(where $\dot s = 0 = \dot p_s$) is independent of that rescaling. This
paper proves that

\begin{theorem}
  \label{thm:kam-tori-for-nose-like-thermostats}
  Let $(N,u) = (N_T(s,p_s), u(s))$ be a thermostat that satisfies
  \begin{enumerate}
  \item $N$ is homogeneous quadratic and increasing in $p_s$;
  \item $u : \R^+ → \R^+$ is an increasing diffeomorphism;
  \item for all Hamiltonians $H=H(q,p)$, if $F=H(q,p/u(s)) +
    N_T(s,p_s)$ has a thermodynamic equilibrium then it is independent
    of $s$.
  \end{enumerate}
  Then, up to a rescaling and change of variables, $u = s$ and there
  is a smooth positive function $Ω_T = Ω_T(u)$ such that
  \begin{equation}
    \label{eq:nose-like-thermostat}
    N = \frac{1}{2} Ω_T p_u^2 + n k T \ln u.
  \end{equation}
  In addition, if $Ω_T(u/√ T) \stackrel{T → ∞}{\longrightarrow} Ω(u)$
  in $C^r(\R^+,\R^+)$ for some $r>4$, then the Nos{\'e}-thermostated
  Hamiltonian $F$ associated to $H$ \eqref{eq:h} possesses invariant
  KAM tori for all $T > 0$ sufficiently large.
\end{theorem}

This theorem is proven in a manner similar to that of Theorem
\ref{thm:kam-tori-high-temperature-limit}. Indeed, Theorem
\ref{thm:kam-tori-high-temperature-limit} can be viewed as a special
case of \ref{thm:kam-tori-for-nose-like-thermostats}.

\subsection{A Hamiltonian Proof of Non-Ergodicity of the Thermostated
  Harmonic Oscillator}
\label{sec:ham-proof}

It is common in the analysis of the Nos{\'e}-Hoover thermostat to fix
the temperature $T=1$ and let the thermostat mass $M → ∞$ (the
weak-coupling limit). This is not equivalent to fixing the thermostat
mass $M=1$ and letting $T → ∞$ (the high-temperature limit), see
\eqref{eq:rescaling} below, but the method used in the proof of
Theorem \ref{thm:kam-tori-high-temperature-limit}, along with
first-order averaging, yields a proof of the following theorem, first
proven in \cite{MR2299758}.

\begin{theorem}
  \label{thm:kam-tori-large-mass-limit}
  Let $ω > 0$ and
  \begin{equation}
    \label{eq:h-ho}
    H(q,p) = \frac{1}{2} p^2 + \frac{1}{2} (ω q)^2.
  \end{equation}
  Fix the temperature $T>0$. The Nos{\'e}-thermostated Hamiltonian $F$
  \eqref{eq:nose} associated to $H$ possesses KAM tori for all $ε=1/√M
  > 0$ sufficiently small.
\end{theorem}

\section{Terminology and Notation}
\label{sec:term-not}

Generating functions provide a convenient way to create canonical
transformations. To explain, let $(q',p') = f(q,p)$ be a canonical
transformation, so that $q' · dp' + p · dq = d φ$ is closed and
therefore locally exact. That is, there is a locally-defined function
$φ = φ(p';q)$ of the mixed coordinates $(p';q)$ such that $q' = ∂ φ/∂
p'$ and $p = ∂ φ/∂ q$. The transformation $f$ is implicitly determined
by $φ$. The identity transformation has the generating function $φ = q
· p'$.

In the sequel, a canonical system of coordinates $(x,X) = (x_1 , …,
x_n, X_1, …, X_n)$ are denoted using the capitalization convention:
the Liouville $1$-form equals $\sum_{i=1}^n X_i \D x_i$ and $X_i$ is
the momentum conjugate to the coordinate $x_i$.

The KAM theorem gives sufficient conditions which imply that a
sufficiently smooth perturbation (say $C^r$ for $r>2n$) of an integrable
$n$-degree of freedom Hamiltonian has invariant tori. A Hamiltonian
which satisfies one of these sufficient conditions is said to be KAM
sufficient.

In practice, construction of action-angle coordinates for a particular
Hamiltonian is a very difficult problem. However, \textem{approximate}
action-angle coordinates may be constructed by methods similar to
their construction in the Birkhoff Normal Form: by means of a sequence
of generating functions that transform the Hamiltonian into a
near-integrable form. In this case, one verifies KAM sufficiency for
the integrable approximation.

\section{The Rescaled Thermostat}
\label{sec:proof}

Let us rescale the variables in the Nos{\'e} thermostat so that the
Boltzmann constant $k=1$ and
\begin{align}
\label{eq:rescaling}
q  &=  \sqrt{M}\, w \bmod 2 π, &&& p  &=  W / \sqrt{M}, &&& s  &=  σ/\sqrt{MT}, &&& p_s  &=  \sqrt{MT}\, Σ.
\end{align}
With this canonical change of variables, the thermostated Hamiltonian
for $H$ \eqref{eq:h} is ($ε = 1/\sqrt{M}$)
\begin{equation}
  \label{eq:rescaled-f}
  F = T \times \underbrace{\left[ \frac{1}{2} \left( W/σ\right)^2 + \frac{1}{2} Σ^2 + β V(w/ε) + \ln σ \right]}_{F_{β}} - \frac{1}{2} T \ln(MT).
\end{equation}
Since the coordinates $(w,σ)$ and $(W,Σ)$ are canonically conjugate,
up to a rescaling of time by the factor $T$, the Hamiltonian flow of
$F$ equals that of $F_{β}$.

\section{KAM tori in the high-temperature limit}
\label{sec:kam}

Because the timescale of the thermostat, $ε$, enters into the rescaled
thermostated Hamiltonian $F_{β}$ only through the bounded potential
$V$, and the analysis of this section focuses on the high-temperature
limit $β → 0 ⁺$, the convention is adopted that
\begin{equation}
  \label{eq:eps=1}
  M = 1 \qquad (\implies ε = 1).
\end{equation}
The analysis below is altered in insignificant ways by this additional
hypothesis.

\begin{lemma}
  \label{rot-inv}
  Let $β = 0$. Under the canonical change of coordinates induced by introducing cartesian coordinates,
  \begin{equation}
    \label{eq:polar}
    (a,b) = ( σ \cos w, σ \sin w ),
  \end{equation}
  the rescaled thermostated Hamiltonian equals
  \begin{equation}
    \label{eq:f0}
    F_0 = \frac{1}{2} \left[ A^2 + B^2 \right] + \frac{1}{2} \ln \left( a^2 + b^2 \right).
  \end{equation}
  That is, $F_0$ is a mechanical hamiltonian with a rotationally invariant potential.
\end{lemma}

The proof is a simple computation. With the interpretation that $F_0$
is the Hamiltonian of the thermostated free particle ($V \equiv 0$),
Hoover \cite{hoover} observed this integral, or rather its reduced
form, and the reduced integral appears in the work of Legoll, Luskin
and Moeckel~\cite{MR2299758,MR2519685}.

There is a family of periodic orbits of $F_0$ along the variety 
\begin{equation}
  \label{eq:periodic-variety}
  Ξ = \set{(σ,w,Σ,W) \mid σ=⎸ W ⎹ ≠ 0, Σ=0},
\end{equation}
with each periodic orbit parameterized by the angular momentum
integral $μ = W$. Ideally, one would like to apply a theorem of
R{ü}ssmann and Sevryuk \cite{MR1354540,MR1390625}. In this context the
theorem says that if the ratio of periods $T_1/T_2$ of the periodic
orbit and the linearized reduced hamiltonian is not constant, then
$F_0$ is KAM-sufficient, i.e. invariant KAM tori survive for $F_{β}$
for all $β$ sufficiently small. Unfortunately, the potential functions
$U(σ) = σ^{α}/α$ (including the degeneration, $U=\ln$, at $α = 0$) are
characterized by constancy of this ratio.

Instead, we compute an approximate change of coordinates to
action-angle variables using a succession of generating functions.

As noted above, $F_0$ has an invariant family of periodic orbits along
the variety $Ξ$, with each periodic orbit $Ξ_{μ} = \set{(⎸ μ ⎹ ,w,0,μ)
  \mid w ∈ \R/2 π \Z}$ parameterized by angular momentum $μ ≠ 0$. On
the other hand, let $\cotangent \T^2$ have the canonical coordinates
$\set{(θ,η,I,J) \mid θ, η ∈ \R/2 π \Z,\ I,J ∈ \R}$ and let $Z ⊂
\cotangent \T^2$ be the zero section $\set{ (θ,η,0,0) }$.

\begin{lemma}
  \label{mu-is-1}
  There are open sets $A ⊂ \cotangent \T^2$, $B ⊂ \cotangent (\R^+ ×
  \T^1)$ such that $Z ⊂ A$, $Ξ_1 ⊂ B$ and a canonical transformation
  $$Φ : A - Z → B - Ξ_1 \qquad (σ, w, Σ, W)= Φ(θ,η,I,J)$$ that transforms the Hamiltonian $F_0$
  \eqref{eq:rescaled-f} to
  \begin{align}
    \label{eq:bnf} 
    F_0 &= I(-\frac{11}{24} I + 1 + J + J^2) - J(1+J/2+J^2/3+J^3/4) + O(5)
  \end{align}
  where $I$ has degree $2$, $J$ has degree $1$ and $O(5)$ is a
  remainder term containing terms of degree $≥ 5$.
\end{lemma}

\begin{remark}
  The transformation $Φ$ extends continuously over the zero section
  $Z$. The extension blows down the $2$-torus $Z$ to the $1$-torus
  (periodic orbit) $Ξ_1$ by collapsing the $θ$-cycle on $Z$. In
  addition, the non-standard choice of degrees for the action
  variables $I$ and $J$ is because they are determined by the pullback
  of the degrees of $σ, w, Σ$ and $W$ (all of degree $1$) by $Φ$.
\end{remark}

\begin{proof}
  The generating function $φ(W,Σ;u,v) = (1-u) W Σ + (1-W) v$ induces the
  canonical transformation $(σ,w,Σ,W) = f(u,v,U,V)$ where
  \begin{align}
    \label{eq:fgen}
    σ &= (1-u)(1-V), &&& w &= -v - U(1-u)/(1-V) \bmod 2 π, \\\notag
    Σ &= U/(V-1),    &&& W &= 1-V.
  \end{align}
  This transforms the Hamiltonian $F_0$ to
  \begin{equation}
    \label{eq:f0-after-f}
    F_0 = \underbrace{\frac{1}{2} (1-u)^{-2} + \frac{1}{2} (1-V)^{-2} U^2 + \ln(1-u)}_{G_0} + \ln(1-V).
  \end{equation}
  The symplectic map $f$ is singular along the set $\set{V=1}$ (which
  should be mapped to the zero angular momentum locus $\set{W=0}$),
  and it transforms $\set{u=0, U=0}$ to the variety of periodic points
  $Ξ$. By design, $f$ is a symplectomorphism of $\set{(u,v,U,V) \mid V
    < 1} ⊂ \cotangent (\R × \R/2 π \Z)$ onto an open neighbourhood of
  $Ξ$. Additionally, $f$ maps an open neighbourhood of
  $\set{(u,v,U,V) \mid u=U=V=0}$ onto an open neighbourhood of the
  periodic locus $Ξ_1$.

  The determination of a further coordinate change is independent of
  the final term in $F_0$, which involves only $V$, so let $G_0 = F_0
  - \ln(1-V)$ as indicated in \eqref{eq:f0-after-f}. With the fourth-order
  Maclaurin expansion of $G_0$, one obtains
  \begin{equation}
    \label{eq:f0-almost-bnf}
     G_0 = \left(\frac{3V^2}{2}  + V + \frac{1}{2}\right) U^2 + \left(\frac{9u^2}{4} + \frac{5u}{3}  + 1\right) u^2  + O(5),
  \end{equation}
  where $O(5)$ is the remainder term that contains terms of degree $5$ and higher.

  One postulates a second generating function
  \begin{equation}
    \label{eq:nu-gen-fun}
    ν = ν(U,V;x,y) = xU + yV + \sum_{3 ≤ i+j+k+l ≤ 4} ν_{ijkl} x^i y^j U^k V^l + O(5),
  \end{equation}
  and a transformed Hamiltonian\footnote{A reader who is familiar with the Birkhoff normal form may wonder why $G_0$ includes cubic terms. These computations mirror those for the Birkhoff normal form, but our Hamiltonian is not being expanded in a neighbourhood of an isolated critical point.}
  \begin{equation}
    \label{eq:g0-nf}
    G_0 = \left(x^2 + \frac{X^2}{2} \right)\left(α\left(x^2 + \frac{X^2}{2} \right) + γ Y^2 + βY + 1\right) + O(5).
  \end{equation}
  One solves for the generating function $ν$ and $G_0$ simultaneously, and arrives at
  \begin{eqnarray}
    ν &= yV + \frac{55 U x^3}{144}  - \frac{5 U V x^2}{6}  - \frac{5 U x^2}{6}  + \frac{3 U V^2 x}{8}  + \frac{U V x}{2}  + \frac{233 U^3 x}{288}\\    \label{eq:nu-gen-fun-soln}
    &+ xU - \frac{5 U^3 V}{9}  - \frac{5 U^3}{18} + O(5) \notag
  \end{eqnarray}
  and $α = -11/24$, $β = γ = 1$.

  Finally, let $I = (x^2 + X^2/2)$, $θ$ be the conjugate angle ($\bmod
  2 π$), and $η = y \bmod 2 π$, $J = Y$. Then the transformed
  Hamiltonian $F_0$ is congruent $\bmod\ O(5)$ to that in
  \eqref{eq:bnf}.
\end{proof}

\begin{proof}[Proof of Theorem \ref{thm:kam-tori-high-temperature-limit}]
  The rescaled thermostated Hamiltonian $F_{β} = F_0 + β V(q) = F_0 +
  O(β)$ where $O(β) = β V(w)$ is $C^r$, $r>4$, and $2 π$-periodic in
  $w$. Under the sequence of canonical transformations in lemma
  \ref{mu-is-1}, $w = -η + ρ(θ,η,I,J) + O(5) \bmod 2 π$ where $ρ$ is
  an analytic real-valued function, and $O(5)$ is a remainder in
  $I,J$. So the perturbation in the approximate angle-action variables
  $(θ,η,I,J)$ is $C^r$, $r>4$, and $O(β)$.

  Since $F_0$ \eqref{eq:bnf} has a non-vanishing Hessian determinant
  in the action variables $(I,J)$, the KAM theorem
  applies~\cite{MR2554208,MR1858536,MR0163025,MR0147741}.
\end{proof}

\section{Nos{\'e}-like Thermostats}
\label{sec:nose-like-thermostats}

This section proves theorem
\ref{thm:kam-tori-for-nose-like-thermostats}. This section employs the
convention that $G_i$ denotes the partial derivative of the function
$G$ with respect to the $i$-th variable.

\subsection{The Thermostat's Normal Form}
\label{sec:normal-form}

To prove the normal form for a Nos{\'e}-like thermostat in
\ref{thm:kam-tori-for-nose-like-thermostats}, observe that Hamilton's
equations for the Hamiltonian $F(q,p,s,p_s) = H(q,p/u) + N(s,p_s)$ are
\begin{align}
  \label{eq:h-eqns-nlike}
  \dot q &= u^{-1} H_2, &&& \dot p   &= -H_1, \\\notag
  \dot s &= N_2,        &&& \dot p_s &= \frac{u'}{u} E(H) - N_1,
\end{align}
where $H_i$ ($N_i$) is the partial derivative of $H$ ($N$) with
respect to the $i$-th argument, $E(H)_{(q,p)} = p · H_2(q,p)$ is the
fibre derivative of $H$ and $H$ and its derivatives are evaluated at
$(q,p/u)$.

In thermodynamic equilibrium, $\dot s = 0 = \dot p_s$. Solving $\dot
p_s = 0$ yields $E(H) = N_1 / (\ln u)_s$. Since the right-hand side is
independent of $(q,p)$, the left-hand side must be depend only on $s$
and therefore it must be constant. Following convention, let $nkT$ be
this constant. Then, since $\dot s = 0$ and $N$ is increasing and
homogeneous of degree $2$ in $p_s$,
\begin{equation}
  \label{eq:n-0}
  N(s,p_s) = \frac{1}{2} A p_s^2 + nkT \ln u,
\end{equation}
where $A=A(s) > 0$. Because $T$ is constant, the function $A$ may be
parameterized by $T$ so: $A = A_T$. Since $u$ is a diffeomorphism, the
change of variables $s → u$ gives
\begin{equation}
  \label{eq:n-nf-0}
  N(u,p_u) = \frac{1}{2} Ω_T p_u^2 + nkT \ln u,
\end{equation}
where $Ω_T = A_T · (u')^2$.  This proves the normal form for
the thermostat under the hypotheses of
\ref{thm:kam-tori-for-nose-like-thermostats}.

\begin{remark}
  In the general case where $N_2$ vanishes along $p_s=0$, $A =
  A(s,p_s)$ is a smooth function of both variables. This added
  generality introduces the possibility of multiple thermodynamic
  equilibria at the same temperature, which differ only in the value
  of the momentum $p_s$. It is difficult to understand the
  significance of this.
\end{remark}

\subsection{KAM-tori in the high-temperature limit}
\label{sec:kam-nose-like}

By means of the rescaling in \eqref[1]{eq:rescaling}, with $M=1$, the
thermostated Hamiltonian is transformed to
\begin{equation}
  \label{eq:rescaled-f-nose-like}
  F = T \times \underbrace{\left[ \frac{1}{2} \left( W/σ\right)^2 + \frac{1}{2} Ω_T(σ/\sqrt{T}) Σ^2 + β V(w/ε) + \ln σ \right]}_{F_{β}} - \frac{1}{2} T \ln(T).
\end{equation}
By the hypothesis of
Theorem~\ref{thm:kam-tori-for-nose-like-thermostats}, as $T → ∞$,
$Ω_T(σ/√ T)$ converges in $C^r(\R^+, \R^+)$ to a limit $Ω(σ)$ for some
$r>4$.

The Hamiltonian $F_0$ has the invariant variety
$Ξ$~\eqref{eq:periodic-variety} of periodic points and the invariant
periodic set $Ξ_1$, as in the constant thermostat mass case.

\begin{lemma}
  \label{lem:nose-like-normal-form}
  Assume that $Ω(σ) = 1 + a(σ-1) + b(σ-1)^2/2 + \cdots$. If $b =
  (96α+9a^2-30a+44)/6$, $β = (2-a)/2$ and $γ = (48α+3a^2-21a+34)/12$,
  then  there are open sets $A ⊂ \cotangent \T^2$, $B ⊂ \cotangent (\R^+ ×
  \T^1)$ such that $Z ⊂ A$, $Ξ_1 ⊂ B$ and a canonical transformation
  $$Φ : A - Z → B - Ξ_1 \qquad (σ, w, Σ, W)= Φ(θ,η,I,J)$$ that transforms the Hamiltonian $F_0$
  \eqref{eq:rescaled-f-nose-like} to
  \begin{align}
    \label{eq:bnf-nl}
    F_0 &= I(α I + 1 + β J + γ J^2) - J(1+J/2+J^2/3+J^3/4) + O(5)
  \end{align}
  where $I$ has degree $2$, $J$ has degree $1$ and $O(5)$ is a
  remainder term containing terms of degree $≥ 5$.
\end{lemma}

\begin{remark}
  In the case $a=b=0$, one finds that $α=-11/24$ and $β=1=γ$, which is
  the result of Lemma~\ref{mu-is-1}. Similar to the assumption that
  $M=1$ in the Nos{\'e}-thermostat case, the assumption that the inverse
  mass $Ω(1)=1$ simplifies the statement of Lemma
  \ref{lem:nose-like-normal-form} and the proof of Theorem
  \ref{thm:kam-tori-for-nose-like-thermostats}, but the latter Theorem
  holds for any value of $Ω(1)>0$.
\end{remark}

The proof of Lemma~\ref{lem:nose-like-normal-form} is similar to that
of Lemma~\ref{mu-is-1} and is omitted. The relations between the parameters
$a$ and $b$ of the thermostat's inverse mass $Ω$ and $α, β, γ$ of the
normal form in approximate action-angle variables arise from the
attempt to force the normal form to be $I$.

\begin{proof}[Proof of Theorem \ref{thm:kam-tori-for-nose-like-thermostats}]
  By Lemma \ref{lem:nose-like-normal-form}, the determinant of the Hessian of
  $F_0$ with respect to the action variables $I,J$ is
  \[ - \left( 2 α + β^2 \right) + 4 α γ I - 4 \left( β γ + α \right) J - \left( 4 γ^2 + 6 α\right) J^2 + O(3), \]
  which equals $O(3)$ iff $α=β=γ=0$. However, if $α=0=γ$, then $a ≠ 2$ and so $β ≠ 0$.
\end{proof}

\section{The Harmonic Oscillator in the Weak-Coupling Limit}
\label{sec:ho}

\begin{proof}[Proof of Theorem \ref{thm:kam-tori-large-mass-limit}]
  After applying the change of variables in \eqref[1]{eq:rescaling}, and a
  rescaling of $(W,w)$ the rescaled thermostated harmonic oscillator Hamiltonian is
  \begin{equation}
    \label{eq:rescaled-ho}
    G_{κ} = \frac{1}{2} \left( W/σ\right)^2 + \frac{1}{2} w^2 + κ \left( \frac{1}{2} Σ^2 + \ln σ \right),
  \end{equation}
  where $κ = ε /(ω √ β)$. In the following, it will be assumed that $ω
  = β = 1$ so that $κ = ε$. The generating function $φ = w V √ σ + σ U$ induces the canonical change of variables
  \begin{align}
    \label{eq:change-ho}
    σ &= u, &&& w &= v/√ u, &&& {Σ} &= U + \frac{1}{2} vV/u, &&& W &= V √ u.
  \end{align}
  When composed with the canonical transformation $u → 1-u$, $U → -U$, the Hamiltonian $G_{κ}$ transforms to
  \begin{equation}
    \label{eq:gk}
    G_{κ} = \frac{1}{2(1-u)} \underbrace{\left( V^2 + v^2 \right)}_{2E} + κ \underbrace{\left[ \frac{1}{2} (U - \frac{1}{2} vV/(1-u))^2 + \ln(1-u) \right]}_{Q_{κ}}.
  \end{equation}
  This Hamiltonian weakly couples the variables $(v,V)$ with $(u,U)$
  when $κ \muchlessthan 1$, with $(v,V)$ evolving on a fast time-scale
  and $(u,U)$ evolving on a slow timescale. Averaging the Hamiltonian
  $G_{κ}$ in $(v,V)$ over a period gives
  \begin{equation}
    \label{eq:gk-ave}
    \bar{G}_{κ} = \frac{E}{2 (1-u)} + κ \left[ \frac{1}{2} U^2 + \frac{E^2}{2^6 (1-u)^2} + \ln(1-u) \right] + O(κ^2)
  \end{equation}
  The Hamiltonian ${κ}^{-1} \bar{G}_{κ}$ has a second-order Maclaurin expansion of
  \begin{equation}
    \label{eq:gk-ave-t2}
    \frac{1}{2} U^2 + \left(\frac{3\,E^2}{16}  + \frac{E}{κ}  - \frac{1}{2} \right)\,u^2 + \left(\frac{E^2}{8}  + \frac{E}{κ}  - 1\right)\,u   + \frac{E^2}{16}  + \frac{E}{κ} + O(κ).
  \end{equation}
  When $E = κ + O(κ^2)$, ${κ}^{-1} \bar{G}_{κ}$ has a critical point at $u=U=0$ and the fourth-order Maclaurin expansion is
  \begin{equation}
    \label{eq:gk-ave-t4}
     \frac{1}{2} \left( U^2 + u^2 \right)  + \frac{2\,u^3}{3}  + \frac{3\,u^4}{4} + 1 + O(κ)
   \end{equation}
   Computations similar to those in Lemma \ref{mu-is-1} show that the Birkhoff Normal Form is
   \begin{equation}
     \label{eq:gk-ave-bnf}
     \bar{G}_{κ} = κ I (1 - 13 I/24) + O(κ^2, 5),
   \end{equation}
   where $I = \frac{1}{2} \left( U^2 + u^2 \right)$. Since the averaged system is KAM sufficient, the unaveraged Hamiltonian $G_{κ}$ is an $O(κ^2)$ perturbation of a KAM sufficient Hamiltonian system.
\end{proof}

\begin{remark}
  \label{rmk:ho}
  One may attempt to apply the Birkhoff Normal Form to the Hamiltonian $\hat{G}_{κ} = G_{κ} - κu/(1-u).$
  The generating function
  \begin{align}
    \label{eq:gen-fun-not-bnf}
    ν &= Ux-2Ux^2/3+65Ux^3/288+295U^3x/288-4U^3/9 \\\notag
    &+ Vy + U(x-2)(y^2+V^2)/4κ+U^2Vy/2κ^2
  \end{align}
  induces a canonical transformation $(u,v,U,V) = f(θ,η,I,J)$ that transforms $\hat{G}_{κ}$ to normal form:
  \begin{equation}
    \label{eq:hatgk-bnf}
    \hat{G}_{κ} = κ I + J + α I^2 + β I J + γ J^2 + O(5),
  \end{equation}
  where $α = -13 κ/24$, $β = -1$, $γ = -1/2 κ$ and $I=(x^2 + X^2)/2$,
  $J = (y^2 + Y^2)/2$. Note that when $J=0$, $\hat{G}_{κ}$ in
  \eqref{eq:hatgk-bnf} coincides with the averaged Hamiltonian
  $\bar{G}_{κ}$ in \eqref{eq:gk-ave-bnf}.\footnote{The first line of
    \eqref{eq:gen-fun-not-bnf} provides the generating function to
    transform \eqref{eq:gk-ave-t4} to \eqref{eq:gk-ave-bnf}.}

  If the total energy is fixed at $\hat{G}_{κ} = κ h$, then
  \begin{equation}
    \label{eq:j}
      J = κ  \left( h + h^2/2 - I + I^2/24 + O(5) \right)
  \end{equation}
  is the Hamiltonian of the reduced system $d θ/d η = - ∂ J/∂ I$, $d
  I/d η = 0$ on the isoenergy level $\hat{G}_{κ} = κ h$. This implies
  that the Hamiltonian $\hat{G}_{κ}$ is KAM sufficient
  \cite[pp. 46--47]{MR2554208}.

  The final step in this line of proof would be to prove that in the
  limit at $κ = 0$ of a suitably renormalized $\hat{G}_{κ}$ is KAM
  sufficient and $G_{κ}$ is a suitably small perturbation.
\end{remark}

\section{Conclusion}
\label{sec:conclusion}

This note has demonstrated the existence of KAM tori near the
high-temperature limit of a Nos{\'e}-thermostated $1$-degree of freedom
system with a periodic potential, along with similar thermostats that
are scale-invariant. It has also given a ``Hamiltonian'' proof of
Legoll, Luskin and Moeckel's result on the existence of KAM tori in
the Nos{\'e}-Hoover thermostated harmonic oscillator in the
weak-coupling limit.

It is expected that the techniques of this paper may be used to
demonstrate similar results for $n$-degree of freedom
Nos{\'e}-thermostated systems. Potentially more fruitful, however, is
that the techniques of this paper might be useful to create
thermostats with the desired properties. Of course, some features of
the Nos{\'e}-type thermostat must be abandoned in the process.

\bibliographystyle{amsplain}
\bibliography{nh}
\end{document}

%% file: definitions.tex
\usepackage[citebordercolor={0.7 0.7 0.7}]{hyperref}
\usepackage{comment}
\excludecomment{maximacode}
\excludecomment{arxivabstract}
\usepackage[utf8]{inputenc}

\ifcsname showkeystrue\endcsname
\usepackage[notcite,notref]{showkeys}

\fi

\newcommand{\T}{{\bf T}}
\newcommand{\R}{{\bf R}}

\newcommand{\Z}{{\bf Z}}

\newcommand{\set}[1]{\left\{#1\right\}}

\newtheorem{theorem}{Theorem}[section]

\newtheorem{lemma}{Lemma}[section]
\theoremstyle{definition}

\theoremstyle{remark}

\newtheorem{remark}{Remark}[section]

\newcommand{\safeinput}[1]{\IfFileExists{#1}{{\catcode`\&=0\input{#1}}}{\message{File
        #1 does not exists. Skipping.}}}

\newcommand{\D}[1]{\,\mathrm{d}#1}
\newcommand{\mean}[1]{\overline{#1}}
\newcommand{\bmean}[1]{\hat{#1}}
\newcommand{\textem}[1]{{\em #1}}
\newcommand{\defn}[1]{\textem{#1}}

\ifcsname nhpaper\endcsname
\DeclareUnicodeCharacter{9144}{\left|}
\DeclareUnicodeCharacter{9145}{\right|}
\DeclareUnicodeCharacter{23B8}{\left|}
\DeclareUnicodeCharacter{23B9}{\right|}
\else
\includecomment{arma}
\DeclareUnicodeCharacter{9144}{\left|\hspace{-0.25ex}\left|} 
\DeclareUnicodeCharacter{9145}{\right|\hspace{-0.25ex}\right|}
\fi
\DeclareUnicodeCharacter{183}{\cdot}
\DeclareUnicodeCharacter{189}{\frac{1}{2}}
\DeclareUnicodeCharacter{2026}{\ldots} 
\DeclareUnicodeCharacter{207A}{^{+}} 
\DeclareUnicodeCharacter{2192}{\longrightarrow} 
\DeclareUnicodeCharacter{2202}{\partial} 
\DeclareUnicodeCharacter{2208}{\in} 
\DeclareUnicodeCharacter{221A}{\sqrt} 
\DeclareUnicodeCharacter{221E}{\infty} 
\DeclareUnicodeCharacter{2260}{\neq} 
\DeclareUnicodeCharacter{2264}{\leq} 
\DeclareUnicodeCharacter{2265}{\geq} 
\DeclareUnicodeCharacter{2282}{\subset} 
\DeclareUnicodeCharacter{39E}{\Xi} 
\DeclareUnicodeCharacter{3A3}{\Sigma} 
\DeclareUnicodeCharacter{3A6}{\Phi} 
\DeclareUnicodeCharacter{3A9}{\Omega} 
\DeclareUnicodeCharacter{3B1}{\alpha} 
\DeclareUnicodeCharacter{3B2}{\beta} 
\DeclareUnicodeCharacter{3B3}{\gamma} 
\DeclareUnicodeCharacter{3B5}{\epsilon} 
\DeclareUnicodeCharacter{3B7}{\eta} 
\DeclareUnicodeCharacter{3B8}{\theta} 
\DeclareUnicodeCharacter{3BA}{\kappa} 
\DeclareUnicodeCharacter{3BC}{\mu}
\DeclareUnicodeCharacter{3BD}{\nu} 
\DeclareUnicodeCharacter{3BE}{\xi} 
\DeclareUnicodeCharacter{3C0}{\pi} 
\DeclareUnicodeCharacter{3C1}{\rho} 
\DeclareUnicodeCharacter{3C3}{\sigma} 
\DeclareUnicodeCharacter{3C4}{\tau} 
\DeclareUnicodeCharacter{3C6}{\varphi} 
\DeclareUnicodeCharacter{3C9}{\omega} 
\DeclareUnicodeCharacter{932}{{\mathcal{T}}}
\DeclareUnicodeCharacter{995}{\textPsi}
\DeclareUnicodeCharacter{B7}{\cdot} 
\DeclareUnicodeCharacter{D7}{\times} 

\newcommand{\temp}[1]{{\mathrm T}}
\newcommand{\norm}[1]{⎸ #1 ⎹}

\newcommand{\ddt}[2][\ ]{\frac{\D{#1}}{\D{#2}}}

\newcommand{\cotangent}{T^*}
\renewcommand{\eqref}[2][0]{{\ifnum#1=0(\fi}eq. \ref{#2}{\ifnum#1=0)\fi}}
\newcommand{\muchlessthan}[0]{<\hspace{-1ex}<}

\long\def\aureply#1{\ifhmode\newline\fi\noindent{\bf Author Reply}:\ {\em #1}}

\def\gettimestamp#1<#2>{\def\timestamp{\tt #2}}


%% file: nh.bbl
\providecommand{\bysame}{\leavevmode\hbox to3em{\hrulefill}\thinspace}
\providecommand{\MR}{\relax\ifhmode\unskip\space\fi MR }
\providecommand{\MRhref}[2]{%
  \href{http://www.ams.org/mathscinet-getitem?mr=#1}{#2}
}
\providecommand{\href}[2]{#2}
\begin{thebibliography}{10}

\bibitem{MR0163025}
V.~I. Arnol{$'$}d, \emph{Proof of a theorem of {A}. {N}. {K}olmogorov on the
  preservation of conditionally periodic motions under a small perturbation of
  the {H}amiltonian}, Uspehi Mat. Nauk \textbf{18} (1963), no.~5 (113), 13--40.
  \MR{0163025 (29 \#328)}

\bibitem{MR1858536}
Rafael de~la Llave, \emph{A tutorial on {KAM} theory}, Smooth ergodic theory
  and its applications ({S}eattle, {WA}, 1999), Proc. Sympos. Pure Math.,
  vol.~69, Amer. Math. Soc., Providence, RI, 2001, pp.~175--292. \MR{1858536
  (2002h:37123)}

\bibitem{hoover}
W.~Hoover, \emph{Canonical dynamics: equilibrium phase space distributions},
  Phys. Rev. A. \textbf{31} (1985), 1695--1697.

\bibitem{MR2299758}
{Legoll, Frédéric}, {Luskin, Mitchell}, and {Moeckel, Richard},
  \emph{Non-ergodicity of the {N}os\'e-{H}oover thermostatted harmonic
  oscillator}, Arch. Ration. Mech. Anal. \textbf{184} (2007), no.~3, 449--463.
  \MR{2299758 (2008d:82068)}

\bibitem{MR2519685}
\bysame, \emph{Non-ergodicity of {N}os\'e-{H}oover dynamics}, Nonlinearity
  \textbf{22} (2009), no.~7, 1673--1694. \MR{2519685 (2010c:37141)}

\bibitem{MR0147741}
J.~Moser, \emph{On invariant curves of area-preserving mappings of an annulus},
  Nachr. Akad. Wiss. G\"ottingen Math.-Phys. Kl. II \textbf{1962} (1962),
  1--20. \MR{0147741 (26 \#5255)}

\bibitem{nose}
S.~Nos\'e, \emph{A unified formulation of the constant temperature molecular
  dynamics method}, J. Chem. Phys. \textbf{81} (1984), 511--519.

\bibitem{MR1354540}
M.~B. Sevryuk, \emph{K{AM}-stable {H}amiltonians}, J. Dynam. Control Systems
  \textbf{1} (1995), no.~3, 351--366. \MR{1354540 (96m:58222)}

\bibitem{MR1390625}
\bysame, \emph{Invariant tori of {H}amiltonian systems that are nondegenerate
  in the sense of {R}\"ussmann}, Dokl. Akad. Nauk \textbf{346} (1996), no.~5,
  590--593. \MR{1390625 (97c:58053)}

\bibitem{MR2554208}
Dmitry Treschev and Oleg Zubelevich, \emph{Introduction to the perturbation
  theory of {H}amiltonian systems}, Springer Monographs in Mathematics,
  Springer-Verlag, Berlin, 2010. \MR{2554208 (2011b:37116)}

\end{thebibliography}
